\documentclass[12pt]{amsart}
\usepackage{amsmath,amssymb,amsthm}
\newtheorem{theorem}{Theorem}[section]

\newtheorem{lemma}{Lemma}[section]
\newtheorem{proposition}{Proposition}[section]
\newtheorem{corollary}{Corollary}[section]

\usepackage[colorlinks=true,citecolor=cyan,backref=page]{hyperref}

\usepackage[utf8]{inputenc}

\def\Z{\mathbb{Z}}
\DeclareMathOperator{\GL}{GL}

\DeclareMathOperator{\N}{\mathbb N}
\def\modd#1 #2{#1\ \mbox{\rm (mod}\ #2\mbox{\rm )}}

\DeclareMathOperator{\sgn}{sgn}

\usepackage{tikz}
\usetikzlibrary{patterns}
\usetikzlibrary{fadings}
\tikzfading[name=fade out, inner color=transparent!0, outer color=transparent!90]
\usepackage[labelfont=normalfont,labelformat=simple]{subcaption}
\usepackage{graphicx}
\usepackage[vcentermath]{youngtab}
\usepackage{ytableau}

\usepackage{enumerate}

\usepackage{todonotes}

\newcommand{\genlegendre}[4]{%
  \genfrac{(}{)}{}{#1}{#3}{#4}%
  \if\relax\detokenize{#2}\relax\else_{\!#2}\fi
}
\newcommand{\legendre}[3][]{\genlegendre{}{#1}{#2}{#3}}

\title{Christoffel matrices and Sturmian determinants}

\author[C. Reutenauer]{Christophe Reutenauer}
\address{Christophe Reutenauer,
D\'epartement de math\'ematiques, Universit\'e du Qu\'ebec \`a Montr\'eal}
\email{Reutenauer.Christophe@uqam.ca}

\author[J. Shallit]{Jeffrey Shallit}
\address{Jeffrey Shallit, School of Computer Science, University of Waterloo, Waterloo, Ontario N2L 3G1, Canada}
\email{shallit@uwaterloo.ca}

\date{\today}

\begin{document}

\begin{abstract} 
We discuss certain matrices associated with Christoffel words, and show that they have a group structure.   We compute their determinants and show a relationship with the Zolotareff symbol from number theory.
   \end{abstract}

\maketitle

\tableofcontents

\section{Introduction}

Christoffel words, and their cyclic conjugates, the standard words, are a finitary version of Sturmian sequences. The 
Burrows-Wheeler matrix of such a word, that is, the square tableau whose rows are all the conjugates of such a word, 
lexicographically sorted, is called a Christoffel matrix when the letters of the word are elements of a field.
By construction, Christoffel matrices have two distinct entries.

A remarkable property of the Christoffel matrices of order $n$ over the field $K$ is that they form a commutative subgroup 
of $GL_n(K)$ (Theorem \ref{group}). 

In order to prove this result, and to give the structure of this group, we note that the Christoffel matrices  have the vector 
$(1,\ldots,1)$ as left eigenvector, since the row- and column- sums of 
each matrix in the group are constant. They have as supplementary subspace the space spanned by the differences $e_{i-1}-
e_i$ ($e_i$ the canonical basis), and the restriction to this subspace belongs to a finite commutative group. 
The group is isomorphic to the cartesian product of the square of the group of invertible elements of the ground field by the group of 
invertible elements of the ring $\Z/n\Z$. This allows to compute the inverse of any Christoffel matrix, Corollary \ref{inverse}.
Another consequence is 
that the determinant of such a matrix is equal to $\pm$ its row-sum, multiplied by the difference of its two distinct entries raised to 
the power $n$ (Corollary \ref{detBW}). The sign is the Zolotareff (Zolotarev) symbol, a generalization of the Jacobi symbol.

The next task of this article is to compute the determinants obtained from Sturmian sequences. One takes such a sequence, $s$ 
say, on 
the alphabet $\{0,1\}$ and forms the determinant whose rows are $n$ factors (connected subwords) of this sequence; recall that 
such a sequence has $n+1$ distinct factors of length $n$, so one obtains $n+1$ determinants, which are the components of some 
vector $V_n$ of length $n+1$. It turns out that this vector, viewed as a word, is a {\em perfectly clustering word} on a three-letter 
alphabet, which may be computed precisely (Theorem \ref{main}).

Recall that perfectly clustering words on $\ell$ letters are equivalently defined as {\em word-encodings} of {\em symmetric discrete 
exchanges of $\ell$ intervals}, as follows from the work of Ferenczi and Zamboni \cite{FZ}. Perfectly clustering words were studied 
by Simpson and Puglisi \cite{SP} and by Restivo and Rosone \cite{RR}. When $\ell=2$, by a result of Mantaci, Restivo and 
Sciortino \cite{MRS},  these words are the Christoffel words, and their conjugates, including the {\em standard words} of Aldo de 
Luca. One property of perfectly clustering words that we use is that such a word has a unique palindromic factorization (product of 
two palindromes), a result due to \cite{SP}, and proved previously for Christoffel words by Chuan \cite{Ch}; this property is 
preserved by conjugation.

In order to compute $V_n$, we show that for arbitrarily large values of $n$, precisely $n=|w|-1$, where $w$ is a Chistoffel word 
which 
is a factor of $s$, $V_n$ is a Christoffel word that may be computed using our result on Burrows-Wheeler matrices: it is the last 
row of the  inverse of the Burrows-Wheeler matrix of $w$. Then we describe how to pass from $V_n$ to $V_{n-1}$, up to a sign. This is 
done in two 
ways. One of them is to use the palindromic factorization of  $V_n$: write a product of words $V_n=XxyY$, $Xx, yY$ 
palindromes, with $x,y$ letters; then $V_{n-1}$ is obtained by merging and adding $x$ and $y$ (Corollary \ref{VV}).

In Section \ref{cont-frac}, Theorem \ref{main} is given an interpretation in terms of continued fractions; the Sturmian determinants are expressed using the 
semi-convergents and convergents of the slope of the Sturmian sequence.

In Section \ref{fibo}, we apply the article to compute the determinants associated with the Fibonacci sequence (Fibonacci word), 
which are shown to 
be Fibonacci numbers (up to a sign). The sign is precisely computed in Section \ref{sign0}; it amounts to determine the Zolotareff 
symbol $\legendre{F_{m-2}}{F_{m}}_Z$, that is, the sign of the permutation of $\Z/m\Z$ which is multiplication by $F_{m-2}$ modulo $F_m$, Proposition \ref{sign}.

\medskip

\noindent {\bf Acknowledgments}: this article arose after mail exchanges with Jean-Paul Allouche. We heartily thank him for his 
suggestions and 
encouragements.

\medskip

\noindent
{\bf Convention}: We consider words whose letters are elements of a field, and therefore identify words with vectors over this field. We index rows 
and columns of matrices starting from 0; likewise for vectors and words.

\part{Group of Christoffel matrices}

\section{Burrows-Wheeler matrix of a Christoffel word}\label{BW}

We consider words over an alphabet $A$. Let $w=a_0\cdots a_{n-1}$ be such a word, of length $n$, so that $a_i\in A$. 
We assume that $w$ is {\em primitive}, that is, not the power of another word. Then $w$ has $n$ distinct rotations called {\em conjugates}, of the form $a_i\cdots a_na_0\cdots a_{i-1}$ for $i=0, 
\ldots,n-1$.   For example, the conjugates of the French word
{\tt rai} are $\{ {\tt rai}, {\tt air}, {\tt ira} \}$.

The length of a word $w$ is written $|w|$.  The number of occurrences of a letter $a$ in $w$ is
written $|w|_a$, and called the $a$-{\em length}.

We consider $n \times n$ matrices.  We index
their rows and columns by $\{0,1,\ldots,n-1\}$, and we identify this set naturally with $\Z/n\Z$.

The {\em Burrows-Wheeler matrix} (or BW-matrix for short), due to \cite{BW}, of a primitive word $w=a_0\cdots a_{n-1}$ over a totally ordered alphabet, is the $n \times n$ matrix 
$(a_{ij})$, where for $i=0,\ldots,n-1$, the  words $a_{i,0}\cdots a_{i,n-1}$ are the distinct conjugates of 
$w$, in decreasing lexicographical order\footnote{In contrast with what is commonly used in the literature, where one takes the increasing order.}. See Figure \ref{BWmatrix}.
\begin{figure} [ht]
$$
\left[
\begin{array}{cccccccccc}
1&0&0&1&0&0&0\\
1&0&0&0&1&0&0\\
0&1&0&0&1&0&0\\
0&1&0&0&0&1&0\\
0&0&1&0&0&1&0\\
0&0&1&0&0&0&1\\
0&0&0&1&0&0&1\\
\end{array}
\right]
$$
\caption{Burrows-Wheeler matrix of the Christoffel word $0001001$.}\label{BWmatrix}
\end{figure}

%It follows from the definition that the row sums and the column sums of the BW matrix of a word $w$ are all equal to the sum of the letters in $w$.

Among the many equivalent definitions of Christoffel words (see, for example, the books \cite{L,Re}), we take the following, which actually is a characterization due 
to Mantaci, Restivo and Sciortino \cite{MRS}: $w$ is a {\em lower (resp., upper) Christoffel word} if it is a word over a totally ordered two-letter alphabet $\{a<b\}$ such that
\begin{itemize}
\item[(a)] $|w|_a$ and $|w|_b$ are relatively prime;
\item[(b)] $w$ is the lowermost row (resp., uppermost row) of its BW-matrix; and
\item[(c)] the last column of this matrix is (not necessarily strictly) increasing, from the first row to the last.
\end{itemize}

Of course, we identify naturally rows and columns with the corresponding word. Note that the two previous Christoffel words, the 
first row and the last row, are conjugate. See Figure \ref{BWmatrix}. It is known that $|w_a|$ and $|w|_b$, together with the order $a<b$, completely determine each of these two words. Note that in each row, and each column, there are $|w|_a$ ($resp. |w|_a$) occurrences of letter $a$ (resp. $b$).

\section{Group of Christoffel matrices}

Let $\mathbb K$ be a field. We identify words over $\mathbb K$, in a natural way, with row vectors. For instance, the word $1^n$ 
is identified with the vector of length $n$ each component of which is equal to 1.

We say that a matrix is a {\it Christoffel matrix\/} if it is the  BW matrix of some Christoffel word whose 
alphabet is a two-element subset of $\mathbb K$, where this subset is equipped with some total order. We let $M_n(a,b,r)$ denote the Christoffel matrix of 
length $n$ associated with the lower and 
upper Christoffel word over the two-letter alphabet $\{a,b\}\subset \mathbb K$, with the order $a<b$, with $b$-length $r$. For 
example, $M_7(0,1,2)$ is the matrix represented in Figure \ref{BWmatrix}. %Clearly, the row sum of this matrix is $qa+rb$, where 
%$q$ is the $a$-length and $r$ the $b$-length of this word, so that $q+r=n$.

We now fix $n$ and let $G_n$ denote the group of permutations of $\Z/n\Z$ (where the latter is identified with $[n]=\{0,1,\ldots,n-1\}$) of the form $\omega_r:x\mapsto rx$, for some $r$ relatively prime with $n$. This group is isomorphic to the group of invertible elements of the ring $\Z/n\Z$, under the isomorphism $r\mapsto \omega_r$.

\begin{theorem}\label{group} Suppose that the characteristic of $\mathbb K$ is $> n$. The set of all matrices $M_n(a,b,r)$, where $a,b\in \mathbb K$, $a\neq b$, $r\in[n]$, $r,n$ relatively prime,
and $(n-r)a+rb\neq 0$,
forms a commutative subgroup $\mathcal G$ of $\GL_n(\mathbb K)$.  The mapping $\mathcal G\to K^*\times K^*\times G_n, M_n(a,b,r)\mapsto 
((n-r)a+rb, b-a,\omega_r)$, is an isomorphism.
\end{theorem}

Note that $n-r$ is the number of 
$a$'s and $r$ the number of $b$'s in the  Christoffel word corresponding to the BW matrix $M_n(a,b,r)$. Thus $(n-r)a+rb$ is the  common value of the row-sums, and also of the column-sums, of this matrix. 

We need the following lemma. 

\begin{lemma} Let $M$ be the BW matrix of the Christoffel word $w$ of length $n$ over the alphabet $\{a<b\}$, with $|w|_b=r$. 
Then the word obtained by reading down the first column of $M$ is $b^ra^{n-r}$, and for $i=2,\ldots,n$, the $i$-th column is 
obtained from the $(i-1)$-th column by cyclically shifting the latter downwards by $r$ steps.
\label{thelemma}
\end{lemma}

See, for example, Figure \ref{BWmatrix}, where $r=2$.

\begin{proof} We adapt the definition of lower Christoffel words to upper Christoffel words (see \cite{L,Re}).   Let $q=n-r$. Then the upper Christoffel word $w=a_0\cdots a_{n-1}$ with $|w|_b=r$ is defined by 
$$
a_j=\begin{cases} 
	       b, 	 & \text{if} \,\,  qj \bmod n \in \{ 0,\ldots,r-1\}; \\
	       a,        &  \text{if} \,\, qj \bmod n \in \{  r,\ldots,n-1\}.
         \end{cases}
$$

The conjugates of $w$ are, in decreasing lexicographical order, the words
$$
w_i=a_{i,0}\cdots a_{i,n-1}, i=0,\ldots,n-1,
$$
with
$$
a_{i,j}=\begin{cases} 
	       b, 	 & \text{if} \,\,  (i+qj) \bmod n \in \{ 0,\ldots,r-1\};  \\
	       a,        &  \text{if} \,\, (i+qj) \bmod n \in \{  r,\ldots,n-1\}.
         \end{cases}
$$
Thus the BW matrix $M$ of $w$ is the matrix $(a_{ij})_{0\leq i,j\leq n-1}$.
In particular, $w_0$ is an upper Christoffel word, and the corresponding lower Christoffel word is $w_{n-1}$. 

The formulas show that the first column of $M$, viewed as a word, is $b^ra^{n-r}$. Moreover, that each column is a conjugate of this word. 

Let $j\in\{1,\ldots,n-1\}$.
Since $i+qj\equiv \modd{0} {n}$ is equivalent to $i\equiv \modd{rj} {n}$,
the $r$ occurrences of $b$ in column $j$ are located in rows $rj,rj+1,\ldots, rj+r-1$, viewed modulo $n$. These occurrences are cyclically consecutive, which 
implies the lemma.
\end{proof}

We deduce a well-known and remarkable property of the BW matrix of a Christoffel word; namely, how one passes from each row to the next one. 

In what follows, we let  $r^*$ denote the unique inverse in $\{1,\ldots,n-1\}$ of $r$ modulo $n$, for $0<r<n$ and $r,n$ relatively prime. 

\begin{corollary}\label{consecutiveRows} Suppose that $n\geq 2$. Let $1 \leq i < n$. Then the $(i-1)$-th and $i$-th rows of $M$ are equal, except for the elements in the submatrix
$$
\left[
\begin{array}{cc}
a_{i-1,j-1}&a_{i-1,j}\\
a_{i,j-1}&a_{i,j}
\end{array}
\right]
=
\left[
\begin{array}{cc}
b&a\\
a&b
\end{array}
\right]
$$
where $j=ir^*\bmod n$.
\end{corollary}

See Figure \ref{BWmatrix}, where $r=2$, $r^*=4$, and  $n=7$.

\begin{proof} By inspection of the matrix, using Lemma~\ref{thelemma}, we see that two consecutive rows differ only by a submatrix as indicated (see Figure \ref{BWmatrix}). It remains to determine its exact location. Let $i-1 $ and $i$ be the indices of the two rows. Then the right lower entry of this $2 \times 2$ submatrix is in position $i,j$ say, and in this position there is a $b$ that is the first $b$ of the cyclic run of $b$'s in column $j$.

In the proof of Lemma~\ref{thelemma}, we have seen that in column $j$, the first $b$ in the cyclic run of $b$'s is in row $jr \bmod n$. Hence, modulo $n$, we have $i\equiv jr$, and thus $j\equiv ir^*$.
\end{proof}

\begin{proof}[Proof of Theorem \ref{group}] Let square matrices act at the right on row vectors of length $n$.
Recall that the row-sums of $M=M_n(a,b,r)$ are equal to its column-sums, and to $(n-r)a+rb$. Since the column-sums of $M$ are 
constant, the row vector $1^n$ is an eigenvector of this matrix, with eigenvalue $(n-r)a+rb$.
Since the characteristic is $>n$, the line spanned by this vector has as supplementary subspace the  subspace spanned by the $n-1$ vectors $v_i=e_{i-1}-e_i$, $i=1,\ldots,n-1$, where 
$e_0,\ldots,e_{n-1}$ is the canonical basis of the space of row vectors.

Now $v_iM=(e_{i-1}-e_i)M=L_{i-1}-L_i$, where $L_k$ is the $k$-th row of $M$. It follows from Corollary \ref{consecutiveRows} that $L_{i-1}-L_i=(b-a)(e_{j-1}-e_j)$ 
where $j\equiv \modd{ir^*} {n}$. Hence $v_iM=(b-a)v_{ir^* \bmod n}$. 

In the basis $v_1,\ldots,v_{n-1},1^n$, the matrix of the endomorphism defined by right multiplication by $M$ is a diagonal sum of 
two matrices: the product of $b-a$ by a permutation matrix $N$ of 
order $n-1$, and the $1 \times 1$ matrix $(n-r)a+rb$. The matrix $N$ has a 1 in row $i$ and column $ir^* \bmod n$ for $i=1,\ldots,n-1$, 
and 0's elsewhere; equivalently, the 
1's are in row $jr \bmod n$ and column $j$. In other words, $N$ is the matrix of the permutation $x\mapsto rx$ of $(\Z/n\Z)^*$. 

It follows that the mapping of the theorem is a group homomorphism. 

Let $(c,d,\omega)$ in the group $K^*\times K^*\times G_n$. Looking for a preimage $M_n(a,b,r)$, we must have $c=(n-
r)a+rb,d=b-a, \omega=\omega_r$. The last equation determines $r\in [n]$ uniquely, since $r=\omega_r(1)=\omega(1)$. Then we 
must have $b=a+d$, $c=(n-r)a+r(a+d)=na+rd$, hence, by the hypothesis on the characteristic of $\mathbb K$,  $a=(1/n)(c-rd)$ 
and $b=(1/n)(c-rd+nd)=(1/n)(c+(n-r)d)$. We deduce that the mapping is an isomorphism. In particular, $\mathcal G$ is commutative.
\end{proof}

For later use, we state the following result.

\begin{corollary}\label{ba} Let $q=n-r$.  Let $h\equiv \modd{jq} {n}$ for $1 \leq j < n$. Then the row vectors $(a_{h-1,1},\ldots, a_{h-1,n-j-1})$ and $
(a_{h,1},\ldots,a_{h,n-j-1})$ are equal, except for their last two entries, which satisfy $a_{h-1,n-j-1}=b$ and $a_{h,n-j-1}=a$.
\end{corollary}

\begin{proof} Use Corollary \ref{consecutiveRows}, with $i=h$, thus $i=jq$, with equalities modulo $n$. Let $J=ir^*$. Then $J=jqr^*=-j=n-j$. It follows from the Corollary that rows $i-1$ and $i$ of the matrix coincide in positions $0,1,\ldots,J-2$ and that their $(J-1)$-th entry is $b$ for the first and $a$ for the second. Since $J-1=n-j-1$, we get the result.
\end{proof}

\section{Consequences}

\begin{corollary}\label{inverse} The inverse of $M_n(0,1,r)$ is $M_n(-Q/r,1-Q/r,r^*)$, where $r^*\in [n]$ is defined by $rr^*=1+Qn$.
\end{corollary}

\begin{proof} Taking the image of both matrices under the isomorphism of Theorem \ref{group}, we obtain the two elements $(r,1,\omega_r)$ and $((n-r^*)\frac{-Q}{r}+r^*(1-\frac{Q}{r}),1,\omega_{r^*})$. Since the first component of the second triple is $(1/r)(-nQ+r^*Q+rr^*-r^*Q)=1/r$, these two elements are inverse each of another, which implies the corollary.
\end{proof}

\begin{corollary}\label{detBW} The determinant of $M_n(a,b,r)$ is equal to $$((n-r)a+rb)(b-a)^{n-
1}\sgn(\omega_r).$$ Moreover, $(n-r)a+rb$ is an eigenvalue of $M_n(a,b,r)$, and the others eigenvalues are 
products of $b-a$ by $\varphi(n)$-th roots of unity.
\end{corollary}

\begin{proof} By the proof of Theorem \ref{group}, the matrix $M_n(a,b,r)$ is 
conjugate to the diagonal sum two matrices: the $(n-1) \times (n-1)$ matrix $(b-a)P$, and the $1 \times 1$ 
matrix $((n-r)a+rb)$, where $P$ is the matrix of the permutation $\pi:x\mapsto 
rx$ of $(\Z/n\Z)^*$. This permutation has the same sign as the permutation 
$\omega_r$ of $(\Z/n\Z)$, since $\omega_r$ fixes $0$. Hence 
$\det(P)=\sgn(\omega_r)$. The assertion for the determinant follows.

Now, since the order of the group of invertible elements in the ring $\Z/n\Z$ is 
$\varphi(n)$, we have $\omega_r^{\varphi(n)}=1$. Hence $\pi^{\varphi(n)}=1$, which implies that 
$P^{\varphi(n)}=I_{n-1}$. This implies the assertion on the eigenvalues.
\end{proof}

The sign of the permutation $\omega_r: x\mapsto rx$ of $\Z/n\Z$ is known as the {\em Zolotareff symbol}, denoted $\left(\frac{r}{n}\right)_Z$. 
It is known that for $n$ odd, the Zolotareff symbol is equal to the {\em Jacobi symbol} $\legendre{r}{n}$, which by definition 
is the product $\prod_i\legendre{r}{p_i}$, where $n=\prod_ip_i$ is written as a product of primes, and where $\legendre{r}{p}$ is the {\em Legendre symbol}; see \cite{Z, R, B, Ro}.  For $n$ even, see 
Lerch \cite{Le} or Barnes \cite{Barnes}.

The Christoffel words over the alphabet $\{0,1\}$ are particularly meaningful; their BW matrices multiply as follows.

\begin{corollary} One has the product formula
$$M_n(0,1,r)M_n(0,1,s)=M_n(Q,Q+1,R),$$ where $rs=R+Qn$ (Euclidean division by $n$).
\end{corollary}

\begin{proof} One has simply to apply the isomorphism of Theorem \ref{group}. The identity to prove is equivalent to $$(r,1,\omega_r)(s,1,\omega_s)=(nQ+R,1,\omega_R),$$ which is true.
\end{proof}

As an example, the matrix in Figure \ref{BWmatrix} is $M_7(0,1,2)$. Since $2^2=4+7\cdot0$ and $2^3=1+7\cdot 1$, its square and its cube are the matrices $M_7(0,1,4)$ and $M_7(1,2,1)$, which are
$$
\left[
\begin{array}{cccccccccc}
1&1&0&1&0&1&0\\
1&0&1&1&0&1&0\\
1&0&1&0&1&1&0\\
1&0&1&0&1&0&1\\
0&1&1&0&1&0&1\\
0&1&0&1&1&0&1\\
0&1&0&1&0&1&1\\
\end{array}
\right]
\quad \mbox{and} \quad
\left[
\begin{array}{cccccccccc}
2&1&1&1&1&1&1\\
1&2&1&1&0&1&1\\
1&1&2&1&1&1&1\\
1&1&1&2&1&1&1\\
1&1&1&1&2&1&1\\
1&1&1&1&1&2&1\\
1&1&1&1&1&1&2\\
\end{array}
\right],
$$
the Burrows-Wheeler matrices of the lower Christoffel words $0101011$ and $1111112$.

\part{Sturmian determinants}

\section{Determinants of Sturmian factors}\label{determ-sturmian}

We consider Sturmian sequences (also called Sturmian words, but in this article we prefer to reserve the word ``word" for finite words), over 
the alphabet $\{0,1\}$. It is well known that 
each such infinite sequence has, for each natural number $n$, exactly $n+1$ factors (also called {\em subwords}, but for 
some authors, ``subword" has another 
meaning) of length $n$. Taking $n$ of these $n+1$ words, one obtains an $n$ by $n$ matrix with entries in $\{0,1\}$, whose 
determinants we 
want to compute.

To be precise, we need some notation and terminology. Let $A$ be a $(k+1) \times k$ matrix over a commutative ring. 
Write $A=(L_0,\ldots,L_k)$, where the $L_i$ are the rows of $A$.
Define $a_i=(-1)^{k-i}\det(L_0,\ldots,\hat L_i,\ldots, L_k)$. Then we write
$$D(A)=(a_0,\ldots,a_k),$$
and call $D(A)$ the {\em determinantal vector} of $A$.

Now, let $g$ be a Sturmian sequence, order the $n+1$ factors of length $n$ in reverse 
lexicographic order, and form the $(n+1)\times n$ matrix $G_n$ whose rows are 
these words, the largest being the uppermost row. 

We call $D(G_n)$ the {\em $n$-th determinantal vector} of the given Sturmian 
sequence, and denote it $V_n$. Our aim is to compute it. It will turn out 
that $V_n$, viewed as a 
word, is a perfectly clustering word on an alphabet with two or three letters.

For more on Sturmian sequences, see \cite{AS,L}.

\section{Perfectly clustering words and symmetric discrete interval exchanges}\label{PC}

A word is {\em perfectly clustering} if the last column of its Burrows-Wheeler matrix is (not necessarily strictly) increasing from top to bottom. This matrix was defined in Section 
\ref{BW}; recall that the rows of the matrix are 
decreasing from top to bottom. 

A word $w$ is a {\it Lyndon word} if for each 
proper factorization $w=uv$, $w$ is smaller than $vu$ in lexicographical order; equivalently, $w$ of length $n$ is a Lyndon word $w$ if 
and only if it is the last row of its Burrow-Wheeler matrix, and if the latter has $n$ rows.  For more on Lyndon words, see, e.g., \cite{L} or 
\cite{R}.

%\section{Symmetric discrete interval exchange}

Let $n \geq 1$. Let $(c_1,\ldots,c_\ell)$ be a {\it composition} of $n$, that is, an $\ell$-tuple of natural numbers whose sum is $n$ (for convenience, we allow 0's in the composition). We decompose  the interval $$[n]=\{0,1,2,\ldots,n-1\}$$ into  
intervals in two ways: the intervals $I_1,\ldots,I_\ell$ (resp., $J_1,\ldots,J_\ell$) are defined by the condition that they are consecutive and $\mid I_j\mid 
=c_j$ (resp., $\mid J_h\mid=c_{\ell+1-h}$). Let $S_n$ denote the group of permutations of $[n]$. We define the permutation $\sigma\in S_n$
by the condition that its restriction to each $I_h$ is the unique increasing bijection
onto
%??? JS what does "increasingly" mean in the previous sentence.  Can it be omitted?
$J_{\ell+1-h}$. We call such a permutation a {\em symmetric 
discrete interval exchange}\footnote{The word ``symmetric" refers to the fact that 
the intervals are exchanged according to the central symmetry of the set 
$\{1,2,\ldots,\ell\}$, that is, the mapping $h\mapsto \ell+1-h$.}, and we say it is 
{\em associated with the composition} $(c_1,\ldots,c_\ell)$.

A symmetric discrete interval exchange may be equivalently defined by using {\it local translations}.    Indeed, the   permutation $\sigma$ acts  on $[n]$ as   a 
discrete version of an interval exchange %(see, e.g., \cite{FZ}) 
as we now describe in the case of three intervals. Let $\ell=3$ and 
$(c_1,c_2,c_3)$ the composition, with $c_1+c_2+c_3=n$. Then the permutation  $\sigma$ is defined by

\begin{equation}\label{local-translations}
\sigma(i)=\begin{cases} 
	       i+c_2+c_3, 	 & \text{if} \,\,  i\in I_1=\{0,\ldots,c_1-1\};  \\
	       i+c_3-c_1,   &  \text{if} \,\, i\in I_2=\{c_1,\ldots,c_1+c_2-1\};  \\
	       i-c_1-c_2,   & \text{if} \,\, i\in I_3=\{c_1+c_2,\ldots,n-1\}.
               \end{cases}
\end{equation}
In particular, if $c_2=0$, $\sigma$ is translation by $c_3$ in $\Z/n\Z=\{0,1,\ldots,n-1\}$.

As an example, consider the composition $(2,2,5)$ of 9; then $J_1=\{0,1,2,3,4\},J_2=\{5,6\},J_3=\{7,8\}$ and 
\begin{equation}
\sigma%m(J_3)m(J_2)m(J_1)
=\left( \begin{array}{ccccccccc}
0&1&2&3&4&5&6&7&8 \\
7&8&5&6&0&1&2&3&4
%895671234
\end{array}\right).
\label{sigma1}
\end{equation}
The second row of $\sigma$ is obtained by concatenating the intervals $J_3,J_2,J_1$.
%A {\it descent} of a permutation $\sigma\in S_n$ is an $i=1,\ldots,n-1$ such that $\sigma(i)>\sigma(i+1)$; the {\it descent set} of $\sigma$ 
%is the set of its descents, thus it is a subset of $[n-1]$. Clearly, if $\sigma$ is a symmetric discrete interval as above, then then its 
%descent set is $\{c_1,c_1+c_2,\ldots,c_1+\cdots +c_{k-1}\}$. In the example it is $\{2,5\}=\{2,2+3\}$.

%It will be useful to generalize this construction to {\it pseudo-compositions}, that is $k$-tuples of natural numbers, not necessarily positive; for a pseudo-composition of $n$, the previous contruction and observation make sense, and one obtains the same permutation by removing all the 0's in the pseudo-composition.

We call $\sigma$ a {\it circular} symmetric discrete interval exchange if $\sigma$ is a circular permutation, that is, it has only one 
cycle. The example in Eq.~\eqref{sigma1} is circular, since in cycle form we have $$\sigma=(0,7,3,6,2,5,1,8,4)
%(1,8,3,5,7,2,9,4,6)
.$$
Note that there is no simple criterion known for a symmetric discrete interval exchange to be circular, except for the cases $\ell=2$ ($c_1$ and $c_2$ must be relatively prime), 
and the case $\ell=3$ ($c_1+c_2$ and  $c_2+c_3$ must be relatively prime, by a result of Pak and Redlich \cite{PR}\footnote{They also prove that the probability that a 
symmetric discrete exchange of three intervals in $S_n$ is circular tends to $1/\zeta(2)=6/\pi^2$, when $n\to \infty$.}). However, there is a general construction of 
all circular discrete interval exchanges, for all $\ell $, due to Mélodie 
Lapointe \cite{L}.

Let $A$ be a totally ordered alphabet of cardinality $\ell$.
A {\it word encoding} of a circular symmetric discrete interval exchange $\sigma$ is one of the words obtained by replacing  
each number $i$ in one of the cycle forms of $\sigma$ 
by the $j$-th letter in $A$, if $i\in I_j$. %, where $I_1,\ldots,I_k$ are the %successive intervals of $[n]$ of length $c_1,\ldots,c_k$.  
In the example, we have 
$$I_1=\{0,1\},\ I_2=\{2,3\},\  I_3=\{4,5,6,7,8\},$$ and we obtain the 
word $acbcbcacc$, with $A=\{a<b<c\}$. 

Note that we have $n$ choices for the 
cycle form, so we obtain $n$ words from $\sigma$, which clearly form a conjugacy class. %We shall see below that this 
%conjugation class has $n$ elements (which by definition means that $w$ is {\it primitive}; equivalently
%that each of these $n$ words are primitive); this follows indeed from Lemma \ref{lexico} (ii).
The word encoding is called {\it standard\/} if one chooses the cycle form beginning with $1$, as in the example.

\begin{theorem} (Ferenczi and Zamboni \cite{FZ}) \label{3conditions}   Let $w$ be a primitive word $w$ over a totally ordered 
alphabet.  Then the following conditions are equivalent:
\begin{itemize}
\item[(i)] $w$ is a perfectly clustering Lyndon word;

\item[(ii)] $w$ is the standard encoding of some circular symmetric interval exchange.
\end{itemize}
%
%(iii) if $a,a',b,b'$ are letters such that $a<b$ and $a'<b'$, and $u$ is some word, then $aua',bub'$ are not simultaneously circular 
%factors of $w$.
\end{theorem}

Note that if one of the parts is $0$ in the composition $(c_1,c_2,c_3)$, then the corresponding perfectly clustering Lyndon word is a lower Christoffel word.   In particular, we use this when the composition is $(p,0,q)$: then the word is the lower Christoffel word of slope $q/p$ (slope is defined below). 

%The equivalence of (i) and (ii) is from \cite{FZ}, while that of (i) and (iii) is from \cite{DLPPRT}.

\section{Factors of Sturmian sequences}\label{factors}

The {\em slope} of a Christoffel word $w\in \{0<1\}^*$ is by definition $|w|_1/|w|_0$; lower Christoffel words are in bijection with the positive rationals, via their slope. To also obtain 
the words $0$ and $1$ of length $1$, one has to include the slopes $0/1$ and $1/0$, so that the bijection is now onto $\mathbb Q\cup \infty$.

Let $g=(a_n)_{n\in \N}$ be a Sturmian sequence over the alphabet $\{0,1\}$. 
The limit of the sequence $|w|_1/|w|_0$ for $w$ a factor of
$g$, when the length 
$w$ tends to infinity, is called the {\em slope} of this sequence.
Note that we  are not considering the quantity $|w|_1/|w|$, as do many authors; see Section \ref{cont-frac} on continued fractions, for comparing both notions. % : for further use, we call the limit of this sequence the {\em Slope} of $s$. 

It is well known that the set of factors of a given Sturmian sequence $g$ depends only on its slope. %(equivalently, on its Slope). 
We describe this set as follows. 
Consider the infinite path, in the Stern-Brocot tree (see \cite{GKP}), starting from the root, determined by the slope of $g$. This path is coded by a sequence $(\lambda_n)_{n\geq 1}$ 
over the set $\{\ell, r\}$, where $\ell$ 
means ``left" 
and $r$ means ``right". Each node in this path is a positive rational number, and corresponds to some prefix of the sequence; this number 
is the slope of some lower Christoffel word. %; if $p$ is a prefix of the sequence $(\lambda_n)_{n\geq 1}$, we denote by . 
Let $01=w_0, w_1,w_2,\ldots$ be the 
corresponding sequence of Christoffel words. 
Then one has the following: {\em a word is a factor of the Sturmian sequence $g$ if and only if it is a circular factor of some word $w_n$}. Recall that a {\em circular factor\/} of a 
word $u$ is a factor of some conjugate of $u$; equivalently, a prefix of some conjugate of $u$.

Note that the sequence of lengths $|w_n|$ is strictly increasing, with first element $|w_0|=2$. 

Each Christoffel word $w$, of length $\geq 2$, is uniquely the product of two Christoffel words (theorem of Borel and Laubie, see 
\cite{Re} Theorem 2.4.1); and all three are 
simultaneously lower or upper. It is called it the {\em standard factorization of} $w$, and denote it by $w=w'w''$. 

One has the following fact: for each $n$, $w_n$ is equal to the longer of the two words $w'_{n+1},w''_{n+1}$.

\section{Determinants}

In the next result, vectors are identified with words.

\begin{theorem}\label{main} Let $g$ be a Sturmian sequence over the alphabet $\{0,1\}$, with associated sequence of lower 
Christoffel words $01=w_0,w_1,w_2,\ldots$. Let $n\geq 2$.
The $n$-th determinantal vector $V_n$ of $g$ is a perfectly clustering word on a two- or three-letter alphabet. 

Precisely, let $\nu$ the unique integer such that $|w_{\nu-1}|\leq n\leq |w_\nu|-1$, let $N=|w_\nu|$,
and let $\epsilon$ be the sign of the permutation $x\mapsto |w_\nu|_1 x$ of $(\Z/N\Z)^*$.

\begin{enumerate}
\item If $n=|w_\nu|-1$, then $V_n$ is the lower Christoffel word associated with the composition $$(|w''_\nu|,|w'_\nu|),$$ over the alphabet 
$$\{-|w'_\nu|_1,|w''_\nu|_1\},$$ multiplied by $\epsilon$.

\item If $|w_{\nu-1}|\leq n < |w_\nu|-1$, let $i=|w_\nu|-1-n$; then $V_n$ is the perfectly clustering word over the three-letter alphabet 
$$\{-|w'_\nu|_1,|w''_\nu|_1-|w'_\nu|_1,|w''_\nu|_1\},$$ associated with the composition $$(|w''_\nu|-i,i,|w'_\nu|-i),$$ multiplied by $\epsilon(-1)^t$, 
where $$t=\sum_{1\leq j\leq i} (|w_\nu|-j+d_j-(jq\bmod N)),$$
with $d_i$ equal to the number of $j, 1\leq j\leq i$, such 
that $jq\bmod N$ is smaller than $iq\bmod N$.
\end{enumerate}
\end{theorem}

Define $q=|w_\nu|_0$ and $r=|w_\nu|_1$; in particular $q+r=N$. Then $|w''_\nu|=q^*$, $|w'_\nu|=r^*$, the inverses in $[N]$ of $q$ and $r$ modulo $N$ (see 
e.g., \cite[Corollary 14.1.5]{Re}).

As an example, consider $w_\nu=01^201^301^3$. The standard factorization is $w_\nu=w'_\nu w''_\nu$, $w'_\nu=01^201^3, w''_\nu=01^3$, of 
respective lengths $7,4$, and 
respective $1$-lengths $5,3$. For $n=|w_\nu|-1$, the vector $V_{10}$ is, up to a factor $\pm 1$, the lower Christoffel word associated with the composition $(4,7)$; the alphabet is $\{-5,3\}$, and $V_{10}=\pm(-5,3,-5,3,3,-5,3,3,-5,3,3)$.

Now consider $n=|w_\nu|-3$, so that $i=2$. Then $V_8$ is, up to a factor $\pm 1$, the perfectly clustering word associated with 
the composition $(4-2,2,7-2)=(2,2,5)$, over the alphabet $\{-5,-2,3\}$, and $V_8=\pm(-5,3,-2,3,-2,3,-5,3,3)$ (compare to the 
example in Section \ref{PC}).

It is known that each perfectly clustering word has a unique factorization as the product of two palindromes (Simpson and Puglisi 
\cite{SP}, Corollary 4.4). We call this the {\em palindromic factorization}.

\begin{corollary}\label{VV} For any $n\geq 1$, up to multiplication by $\pm$, $V_n$ is obtained form $V_{n+1}$ by considering its unique 
palindromic factor, and adding the last component of its first factor and the first component of its second factor.
\end{corollary}

\section{Proof of Theorem \ref{main}, first part}

For simplicity, denote $w=w_\nu$. Recall that $|w|=N$. The BW matrix of $w$ is $M_N(0,1,r)$. Since $n=N-1$, the matrix $G_n$, which is the matrix of circular factors of 
length $N-1$ of $w$, is obtained from the latter matrix by removing its 
last column. Let $D(F_n)=(a_0,\ldots,a_n)$; it follows that $a_i$ is the cofactor of the element in the last column, row $i$, of 
$M_n(0,1,r)$. Hence, $V_n=D(G_n)$
is the last row of of $M_n(0,1,r)^{-1}$, multiplied by the determinant of $M_n(0,1,r)$. The latter is equal to $r\epsilon$, by 
Corollary \ref{detBW}. Moreover, $M_N(0,1,r)^{-1}=M_N(-Q/r,1-Q/r,r^*)$,  
where $r^*\in [N]$ is defined by $rr^*=1+QN$, by 
Corollary \ref{inverse}. The last row of this matrix is the lower Christoffel 
word associated with the composition $(N-r^*,r^*)$ over the 
alphabet $\{-Q/r, 1-Q/r\}$. Multiplying by $r$, we obtain the lower 
Christoffel word associated with the composition $(N-r^*,r^*)$ 
over the alphabet $\{-Q, r-Q\}$.
 
Thus, to settles the first part of the proof, it is enough 
to show the four equalities $|w''|=N-r^*,|w'|=r^*$, and $-Q=-|w'|_1, r-
Q=|w''|_1$.

To prove the four equalities, we use the fact that the matrix 
$$\left(\begin{array}{cc} |w'|_0 & |w''|_0 \\ |w'|_1 & |w''|_1 
\end{array}\right)$$
has determinant 1 (see, e.g., \cite[Proposition 14.1.2]{Re}). Thus 
$$|w'|_0|w''|_1-|w'|_1|w''|_0=1.$$ This implies that $$(|w'|_1+|w''|
_1)(|w'|_1+|w'|_0)=1+|w'|_1(|w'|_0+|w'|_1+|w''|_0+|w''|_1).$$ That is, 
$|w|_1|w'|=|w'|_1|w|+1$; in other words, $r|w'|=|w'|_1N+1$. 
Thus $r^*=|w'|, Q=|w'|_1$. 
The two other equalities follow from $w=w'w''$.
 
\section{Some preliminary results}

In this section we assume the hypothesis of Theorem \ref{main}, in the second 
case. Recall that $G_n$ is the matrix whose rows are the $n+1$ factors of 
length $n$ of $s$, in decreasing lexicographic order.

For two matrices $A,B$, we write
\begin{equation}\label{A->B}
A\xrightarrow{h}B
\end{equation}
if for some $k$, $A$ is a $(k+1)\times k$ matrix, $B$ is a  $k \times (k-1)$ 
matrix, rows $h-1$ and $h$ of $A$ are equal except for their last entry, 
which is $1$ for row $h-1$ and $0$ for row $h$, and 
$B$ is obtained from $A$ by removing row $h$ from this matrix, and then 
removing column $k-1$, 
the last column. Note that, equivalently, $B$ is obtained by removing row $h-
1$ and the last column. Recall that we index matrix rows and columns starting from 0.

\begin{proposition}\label{FF} 
\leavevmode
\begin{enumerate}
\item $G_n$ is obtained from $G_{N-1}$ by removing the rows $jq\bmod n$, 
$j=1,\ldots,i$ and by keeping only the $n$ first columns, with $i=N-1-n$.

\item Let $h_i=(iq\bmod N)-d_i$, where $d_i$ is the number of $j, 1\leq j\leq 
i$, such 
that $jq\bmod N$ is smaller than $iq\bmod N$. Then
$$G_{n+1}\xrightarrow{h_i}G_n.$$
\end{enumerate}
\end{proposition}

\begin{proof} 
\leavevmode
\begin{enumerate}
\item Call {\em square} a $2\times 2$ submatrix of $M=M_N(0,1,r)$ as the one appearing in Corollary \ref{consecutiveRows}, so that it 
appears in two contiguous rows and two contiguous columns, and it is the identity matrix of order 2.  As indicated in the lemma, these rows are numbered $x-1,x$ and the columns are numbered $y-1,y$ 
with $x\equiv \modd{ry} {N}$, for some $x,1\leq x\leq N-1$. We claim that if a square does not intersect the $n$ first columns, then $x \equiv \modd{qj} {N}$ for some 
$j=1,\ldots,i$. The claim will be proved below.

Consider the matrix $G$ obtained from $G_{N-1}$ by removing the rows $jq\bmod n$, $j=1,\ldots,i$ and by 
keeping only the first $n$ columns. Since $G_{N-1}$ is obtained from $M$ by removing its last column, $G$ is
equivalently obtained from $M$ by removing the rows $jq\bmod n$, $j=1,\ldots,i$ and by keeping only the first $n$ 
columns. In what follows, for simplicity, we index the rows of $G$ by their index in $M$ (equivalently in $G_{N-1}$) before removal (e.g., if we remove row 1 and keep 0 and 
2, the first two rows of $G$ are indexed by $0$ and $2$). Since the rows of $M$ are strictly decreasing from the upper one to the lowest, 
the rows of $G$ are weakly 
decreasing from the upper to the lower. Suppose, in order to get a contradiction, that they are not strictly decreasing. 
Then for some $x'<x$, 
rows $x'$ and $x$ appear in $G$, are successive, and are equal; hence the squares located in rows $x',x'+1,\ldots, x$ do not intersect the first $n$ 
columns of $M$. By the claim, $x\equiv \modd{qj} {N}$ for some $j=1,\ldots,i$, and this contradicts the fact that $G$ contains row $x$.
Thus the $N-i=n+1$ rows of $G$ are strictly decreasing, hence distinct; each row of $G$ being a prefix of a row of $M$, it is a 
circular factor of $w$, hence the rows of $G$ are the $n+1$ circular factors of $w$, and $G=G_n$.

We prove the claim: if a square (with $x,y$ as above) does not intersect the first $n$ columns, then $y-1\geq n$, hence $y\geq n+1$. Hence $x\equiv 
ry$ for 
some $y=n+1,\ldots,N-1$. Since $q+r=N$, and $n+1=N-i$, by letting $j=N-y$, we have $x\equiv ry\equiv -q(N-j)\equiv \modd{qj} {N}$, 
with $j=i,\ldots,1$. 
%Hence row $x$ was 
%removed from $M_N(0,1,r)$, so does not appear in $G$.

\item By Part 1, $G_n$ is obtained from $G_{n+1}$ by removing its row $iq\bmod N$, if the rows are indexed as they are in $G_{N-1}$, and 
then removing the last column. But 
to get their actual index, one has to subtract the number of rows which were removed before from $G_{N-1}$, which is $d_i$.
Note that $G_{n+1}$ without its last column has two successive equal rows, which (in $M$) are indexed $(iq \! \mod N)-1$ and $iq \! \mod N$. 
By Corollary \ref{consecutiveRows}, the last elements of these rows in $G_{n+1}$ are necessarily $1$ and $0$.
Hence we have $G_{n+1}\xrightarrow{h_i}G_n$.
\end{enumerate}
\end{proof}
 
For two vectors $V,W$, we write 
$$
V\xrightarrow{h}W
$$
if for some $k$ we have
$V=(a_0,\ldots,a_k)$ and
$W=(a_0,\ldots,a_{h-2},a_{h-1}+a_{h},a_{h+1},\ldots,a_k)$. That is, $W$ is obtained from $V$ by merging and adding the components ranked $h-1$ and $h$ in $V$.

\begin{proposition}\label{A->D(A)}
If $
A\xrightarrow{h}B$, with the notations of (\ref{A->B}), then $
D(A)\xrightarrow{h} (-1)^{k-h} D(B),
$
with the notation $D(A)$ of Section \ref{determ-sturmian}.
\end{proposition}

\begin{proof}
Write $A=(L_0,\ldots, L_k)$, where the $L_j$'s are the rows of $A$; similarly, $B=(H_0,\ldots,H_{k-1})$. 
Also let $D(A)=(a_0,\ldots,a_k), D(B)=(b_0,\ldots,b_{k-1})$.
We compute within the matrix $A$, omitting mention of, for notational simplicity, the rows before 
$L_{h-1}$ and those after $L_{h}$: 
\begin{align*}
    (-1)^{k-h}(a_{h-1}+a_{h})&=-\det(\ldots,\hat L_{h-1},L_{h},\ldots)+\det(\ldots,L_{h-1},\hat L_{h},
\ldots)\\
&=\det(\ldots,L_{h-1}-L_{h},\ldots)=\det(\ldots,(0,\ldots,0,1),\ldots)\\
&:=e.
\end{align*}
This is the cofactor of the entry in row $h-1$, column 
$k-1$, of $A$; its corresponding submatrix is $B$ without its row of rank $h-1$, so that 
$e=(-1)^{h-1+k-1}\det(H_0,\ldots,\hat H_{h-1},\ldots,H_{k-1})=b_{h-1}$. Hence $b_{h-1}=(-1)^{k-h}(a_{h-1}+a_h)$.

We now show that $b_j=(-1)^{k-h}a_j$ if $j< h-1$. We have
\begin{align*}
(-1)^{k-j}a_j &=\det(\ldots,\hat L_j,\ldots,L_{h-1},L_h,\ldots)\\
&=\det(\ldots,\hat L_j,
\ldots,L_{h-1}-L_h,L_h,\ldots);
\end{align*}
since $L_{h-1}-L_h$ is equal to $(0,\ldots,0,1)$, we see that $(-1)^{k-j}a_j $ is equal to the cofactor of $A_j$ 
(that is, $A$ without its row $j$) in row $h-2$, column $k-1$ ($A$ has $k$ columns, the last is indexed $k-1$). The corresponding 
minor is obtained from $A_j$ by removing row $h-2$ (which is row $h-1$ of $A$) and the last column; this matrix is $B$ without 
its row $j$. Hence 
\begin{align*}
(-1)^{k-j}a_j &=(-1)^{h-2+k-1}\det(\ldots,\hat H_j,\ldots)\\
&=(-1)^{h+k-1}(-1)^{k-1-j}b_j,
\end{align*}
implying that $b_j=(-1)^{k-h}a_j$. 

The verification that $b_j=(-1)^{k-h}a_{j+1}$ for $j>h-1$ is similar and left to the reader. 
\end{proof}

For two words $v,w$, we write 
$$
v\xrightarrow{h}w
$$
if $v,w$ are perfectly clustering words over the alphabet $\{a<b<c\}$, if $v=a_0\cdots a_{h-2}a_{h-1}a_ha_{h+1}\cdots a_k$, $a_{h-1}
=a, a_h=c$, and $w=a_0\cdots a_{h-2}ba_{h+1}\cdots a_k$.   In other words, $w$ is obtained from $v$ by replacing a factor $ac$ by $b$, 
and the indicated $a,c$ are the letters ranked $h-1,h$ in $v$.

\begin{proposition}\label{uu} Let $n=|w_\nu|-1,\ldots, |w_{\nu-1}|-1$, and $i=|w_\nu|-1-n$.
Let $u_n$ denote the perfectly clustering word associated with the composition 
$(|w''_\nu|-i,i,|w'_\nu|-i)$ over the alphabet $\{a<b<c\}$. 
Then for $n=|w_\nu|-2,\ldots, |w_{\nu-1}|-1$
$$u_{n+1}\xrightarrow{h_i} u_n.$$
\end{proposition}

We  need some preparation.
Given a permutation $\sigma$ of $[N]$, and $0\leq k\leq N-1$, we call, as in \cite{BR}, {\it cyclic restriction} of $\sigma$ to the interval 
$[k]$ the permutation of $[k]$ obtained by 
removing in any cycle form of $\sigma$ each number larger than $k-1$ (this does not depend on the chosen cycle form). Note that 
this definition is a discrete version of {\it Rauzy induction} for continuous interval exchanges.

\begin{lemma}\label{restriction0}
\leavevmode
\begin{enumerate}
\item Let $\sigma$ the  circular symmetric discrete exchange of two intervals, associated with the composition $(\gamma,\rho)$, 
$\gamma,\rho$ relatively prime, $\gamma+\rho=N$. 
Let $n=N-1,\ldots,N-\min(\gamma,\rho)-1$, and $i=N-1-n$, so that
$i=0,\ldots,\min(\gamma,\rho)$. Let $\sigma_n$ denote the cyclic restriction of $\sigma$ to $[n+1]$, so that $\sigma=\sigma_{N-1}$. Then $\sigma_n$
is the symmetric 
discrete exchange of intervals, associated with the composition $(\gamma-i,i,\rho-i)$.

\item Let $v=a_0\cdots a_{N-1}$ be the lower Christoffel word over the alphabet $\{a<c\}$ associated with $\sigma$. 
Suppose that $\gamma<\rho$, so that $\min(\gamma, \rho)=\gamma$. For $n=N-1,\ldots,N-\gamma-1$, let $v_n$ denote the standard word 
encoding of $\sigma_n$ over the alphabet $\{a<b<c\}$, so that $v=v_{N-1}$.
Then if $n\leq N-2$, $v_n$ is obtained from $v$ by replacing in it the factor $a_{j\gamma^*-1}a_{j\gamma^*}$ by $b$, for each 
$j=1,\ldots,i$, where 
$i=N-1-n$, and where $\gamma^*\in [N]$ is the inverse of $\gamma$ modulo $N$; moreover, all these factors are equal to $ac$, 
and are not overlapping.
\end{enumerate}
\end{lemma}

\begin{proof}
\leavevmode
\begin{enumerate}
\item This result was proved during the proof of Lemma 2 of \cite{BR}.

\item By definition, $v$ has $\gamma$ occurrences of the letter $a$, which correspond to the numbers $0,1,\ldots,\gamma-1$ in the 
standard cyclic form of $\sigma$, and $\rho$ occurrences of the letter $c$, which correspond to the numbers $\gamma,
\gamma+1,\ldots,N-1$ in it.
Likewise, by
1, $v_n$ has $\gamma-i$ occurrences of letter $a$, $i$ occurrences of letter $b$ and $\rho-i$ occurrences of letter $c$; they 
correspond, in the standard cyclic form of $\sigma_n$,
to the numbers $0,1,\ldots,\gamma-i-1$ for $a$, $\gamma-i,\ldots, \gamma-1$ for $b$, and $\gamma,\ldots,N-i-1$ for $c$.

We know that $\sigma$ is translation by $\rho$ in $\Z/N\Z$: $\sigma(x)=x+\rho$ in $\Z/N\Z$, and in particular $\sigma^e(0)=e\rho$. Since $\rho$ is the opposite of $\gamma$, we have $\gamma^*\rho\equiv \modd{-1} {N}$; hence 
$\sigma^{j\gamma^*}(0)\equiv \modd{N-j} {N}$ and $\sigma^{j\gamma^*-1}(0)\equiv \modd{\gamma-j}  {N}$. Now by 1, the 
standard cyclic form of $\sigma_n$ is 
obtained by removing the numbers $N-j$, $j=1,\ldots,i$ from the cyclic form of $\sigma$. It follows that $v_n$ is obtained from $v$ by first
removing the letters $a_{j\gamma^*}$, which are all equal to $c$ (since for $j=1,\ldots i$, $N-j=N-1,\ldots,N-i=\gamma+\rho-i\geq\gamma$, since $i\leq\gamma<\rho$); and then
by replacing the letters $a_{j\gamma^*-1}$ (which are all equal to $a$ since for $j=1,\ldots,i$, $\gamma-j=\gamma-1,\ldots,
\gamma-i$) by $b$. 

Since these factors are all equal to $ac$, they cannot overlap.
%Since the numbers $N-j$ are all distinct, and similarly for $\gamma-j$, the factors are not overlapping. The lemma follows.
\end{enumerate}
\end{proof}

\begin{corollary}\label{vv} For $i=1,\ldots, \gamma$, let $d_i$ denote the number of $j\leq i$ such that $j\gamma^*\bmod 
N$ is smaller than $i\gamma^*\bmod N$. Then one has   $v_{n+1}\xrightarrow{i\gamma^* \bmod N-d_i}v_n$ for
$n=N-2,\ldots,N-\gamma-1$.
\end{corollary}

\begin{proof} By Lemma \ref{restriction0}, for $i=1,\ldots,\gamma$, one has $v_{i-1}\xrightarrow{e_i}v_i$ for some value of $e_i$, 
which is the position of the $c$ being removed. If one takes into account the previous removals, one obtains the the indicated 
value of $e_i$.
\end{proof}

\begin{proof}[Proof of Proposition \ref{uu}]
We prove it in the case $|w_\nu''|<|w_\nu'|$ (the other case is similar). It is known (Berstel and de Luca, see Corollary 14.5.6 in 
\cite{Re}) that 
$q^*=|w_\nu''|, r^*=|w_\nu'|$. Hence $q^*<r^*$ (recall that $q^*$ is the inverse of $q$ modulo $N$, and that we identify $\Z/N\Z$ and $
[N]=\{0,1,\ldots,n-1\}$).

By construction, $u_{N-1}$ (which corresponds to $i=0$) is the perfectly clustering word over the alphabet $\{a<b<c\}$ 
associated with the composition $(|
w''_\nu|,0,|w'_\nu|)$. In other words, $u_{N-1}$ is the lower Christoffel word over the alphabet $\{a,c\}$ associated with the composition $
(\gamma,\sigma)=(q^*,r^*)$. With the notations of Lemma \ref{restriction0}, we have $u_n=v_n$, and since $\gamma^*=q$, Corollary \ref{vv} implies the proposition, with $h_i$ defined in Proposition \ref{FF}.
\end{proof}

\section{Proof of Theorem \ref{main}, second part, and  Corollary \ref{VV}}

Let $w=w_\nu$. Consider the substitution $\phi$ that maps 
\begin{itemize}
    \item 
$a$ onto $-|w'|_1$,
\item $b$ onto $|w''|_1-|w'|_1$, 
\item $c$ onto $|w''|_1$.
\end{itemize} 
Since $a<b<c$ and $-|w'|_1<|w''|_1-|w'|_1<|w''|_1$, the map $\phi$ maps  
perfectly clustering words onto perfectly clustering words, and preserves the associated composition. Moreover, for any words 
$u,u'\in\{a,b,c\}^*$ and $U,U'\in \{-|w'|_1,|w''|_1-|w'|_1,|w''|_1\}^*$, one has 
\begin{equation}
 u'\xrightarrow{h}u, U'\xrightarrow{h}U, \phi(u')=U' \Rightarrow \phi(u)=U. 
 \label{star}
 \end{equation}
The number of rows of $G_n$ is $k_i=n+1$, with $i=N-1-n$. 
Letting 
$U_n=D(G_n)$, we have for $n=N-2,\ldots, |w_{\nu-1}|-1$:
$U_{n+1}\xrightarrow{h_i}(-1)^{k_i-h_i}U_n$ by Propositions \ref{FF} and \ref{A->D(A)}. Moreover $u_{n+1}\xrightarrow{h_i}u_n$ 
by Proposition \ref{uu}. By the first case of the Theorem, that is $n=N-1$, one has $U_{n-1}=\epsilon \phi(u_{N-1})$.
It follows then by \eqref{star} that $U_n$ is equal to $\phi(u_n)$ multiplied by $\epsilon$ and by $(-1)^t$, where $t=\sum_{1\leq j\leq i}
(k_j-h_j)$. This proves the theorem.

We now prove Corollary \ref{VV}. It is known from \cite{SP} that if $w$ is a perfectly clustering Lyndon words, then $w$ is the 
smallest, and its reversal $\tilde w$, the largest, word in the conjugation class of $w$. If $w=uv$ is the palindromic factorization, 
then clearly $\tilde w=vu$. But it is also known that the position in $w$ of the first letter of $v$ corresponds to the largest digit in the  
standard cyclic form of the symmetric discrete interval exchange associated with $w$ (\cite{FZ} Proof of Theorem 4; a direct 
proof is given in \cite{LaR}, Lemma 2.2). Hence the factor $ac$ in Lemma \ref{restriction0} corresponds to the palindromic factorization of $v$. The corollary follows.

\section{Example}

We take $w_\nu=01101110111$, so that $q=3,r=8$, $w'_\nu=0110111, w''_\nu=0111$. One has $w_{\nu-1}=w'_\nu$, of length 
$7$, and $N=11$ is the length of $w_\nu$.
The value of $n$ varies from $N-1=10$ to $|w_{\nu-1}|-1=6$.
The rows of the matrices $G_{10},G_9,G_8,G_7,G_6$ are, in reverse lexicographic order, the factors of respective lengths 
$10,9,8,7,6$ 
of the given Sturmian sequence, and actually the circular factors of $w_\nu$. In boldface are indicated the entries 1 and 0 of 
$G_{n+1}$ corresponding to the definition of $G_{n+1}\xrightarrow{h_i} G_n$: the two rows of $G_{n+1}$ with the boldfaced 
entries 
coincide up to these entries. The numbers $h_i$ are equal to $(iq\bmod N)-d_i$; the corrective 
term $d_i$ takes into account the rows that are previously removed. One has, for $i=1,2,3,4$, $iq \bmod N =3,6,9,1$, and 
$d_i=0,1,2,0$, so that $h_i=3,5,7,1$, and $$G_{10}\xrightarrow{3} G_9\xrightarrow{5} G_8
\xrightarrow{7} G_7 \xrightarrow{1} G_6.$$

$$G_{10}=\left(
\begin{array}{cccccccccc}
1&1&1&0&1&1&1&0&1&1 \\
1&1&1&0&1&1&0&1&1&1 \\
1&1&0&1&1&1&0&1&1&\bf 1 \\
1&1&0&1&1&1&0&1&1&\bf 0 \\
1&1&0&1&1&0&1&1&1&0 \\
1&0&1&1&1&0&1&1&1&0 \\
1&0&1&1&1&0&1&1&0&1 \\
1&0&1&1&0&1&1&1&0&1 \\
0&1&1&1&0&1&1&1&0&1 \\
0&1&1&1&0&1&1&0&1&1 \\
0&1&1&0&1&1&1&0&1&1
\end{array}
\right),$$
$$
G_9=
\left(
\begin{array}{cccccccccc}
1&1&1&0&1&1&1&0&1 \\
1&1&1&0&1&1&0&1&1 \\
1&1&0&1&1&1&0&1&1 \\
1&1&0&1&1&0&1&1&1 \\
1&0&1&1&1&0&1&1&\bf 1 \\
1&0&1&1&1&0&1&1&\bf 0 \\
1&0&1&1&0&1&1&1&0 \\
0&1&1&1&0&1&1&1&0 \\
0&1&1&1&0&1&1&0&1 \\
0&1&1&0&1&1&1&0&1
\end{array}
\right),$$
$$G_8=
\left(
\begin{array}{cccccccccc}
1&1&1&0&1&1&1&0 \\
1&1&1&0&1&1&0&1 \\
1&1&0&1&1&1&0&1 \\
1&1&0&1&1&0&1&1 \\
1&0&1&1&1&0&1&1 \\
1&0&1&1&0&1&1&1 \\
0&1&1&1&0&1&1&\bf 1 \\
0&1&1&1&0&1&1&\bf 0 \\
0&1&1&0&1&1&1&0
\end{array}
\right),
$$
$$
G_7=
\left(
\begin{array}{cccccccccc}
1&1&1&0&1&1&\bf 1 \\
1&1&1&0&1&1&\bf 0 \\
1&1&0&1&1&1&0 \\
1&1&0&1&1&0&1 \\
1&0&1&1&1&0&1 \\
1&0&1&1&0&1&1 \\
0&1&1&1&0&1&1 \\
0&1&1&0&1&1&1
\end{array}
\right), \quad
G_6=
\left(
\begin{array}{cccccccccc}
1&1&1&0&1&1 \\
1&1&0&1&1&1 \\
1&1&0&1&1&0 \\
1&0&1&1&1&0 \\
1&0&1&1&0&1 \\
0&1&1&1&0&1 \\
0&1&1&0&1&1
\end{array}
\right).
$$
The word $u=u_{10}$ is the lower Christoffel word over $\{a,c\}$ associated with the composition $(|w_\nu''|,|w_\nu'|)=(4,7)$. Hence $u=acaccaccacc$. One has
$$
u_{10}=ac{\bf ac}caccacc\xrightarrow{3}u_9=acbc{\bf ac}cacc$$
$$\xrightarrow{5}u_8=acbcbc{\bf ac}c \xrightarrow{7}u_7=acbcbcbc\xrightarrow{1}u_6=bbcbcbc.$$
%The specific $ac$ that is transformed into $b$ is indicated in boldface. This factor may be found without computing the associated 
%permutation (the symmetric discrete interval exchange): it corrresponds to the factorization of $u_{n+1}$ into two palindromes, as 
%follows from the general theory of perfectly clustering words (see \cite{SP,La}).

The determinantal vectors are correspondingly, up to a factor $\pm 1$:
$$
(-5,3,{\bf -5,3},3,-5,3,3,-5,3,3)\xrightarrow{3}(-5,3,-2,3,{\bf -5,3},3,-5,3,3)$$
$$\xrightarrow{5}(-5,3,-2,3,-2,3,{\bf -5,3},3)\xrightarrow{7}({\bf-5,3},-2,3,-2,3,-2,3)$$
$$
\xrightarrow{1}(-2,-2,3,-2,3,-2,3),
$$
where the two components which are added are indicated in boldface.

As the reader may verify, the transformations between vectors above correspond to the palindromic factorization, as described in 
Corollary \ref{VV}, and the boldfaced components indicate the palindromic factorization. 

For the subsequent determinantal vectors, of smaller length, one has
$$(-2,{\bf -2,3},-2,3,-2,3)\to (-2,1,{\bf -2,3},-2,3)\to (-2,1,1,{\bf -2,3})$$
$$\to ({\bf -2,1},1,1) \to ({\bf -1,1},1) \to ({\bf 0,1})\to (1).
$$

\section{Continued fractions}\label{cont-frac}

Define the matrix $$P(a)=\left( \begin{array}{cc} a&1\\1&0 \end{array}\right).$$
Define the {\em continuant polynomials} $K_n(x_1,\ldots,x_n)$ by $K_0() = 1$, $K_1(x_1) = x_1$ and the recursion 
$K_n(x_1,...,x_n)=K_{n-1}(x_1,...,x_{n-1})x_n +K_{n-2}(x_1,...,x_{n-2})$ if $n\geq 2$. It is well known that
\begin{equation}\label{continuant-matrice}
P(n_0)\cdots P(n_k)=\left( 
\begin{array} {cc}
K(n_0,\ldots,n_k)&K(n_0,\ldots,n_{k-1})\\
K(n_1,\ldots,n_{k})& K(n_1,\ldots,n_{k-1})
\end{array} 
\right),
\end{equation}
where the index of $K$ is omitted, since it is redundant.
Continued polynomials are used to compute the numerators and denominators of finite continued fractions. One has
$$
[n_0,\ldots,n_k]=\frac{K(n_0,\ldots,n_k)}{K(n_1,\ldots,n_{k})}.
$$
In particular, if a Christoffel word $w$ has slope $[n_0,\ldots,n_k]$, then its length is $$|w|=K(n_0,\ldots,n_k)+K(n_1,\ldots,n_{k}).$$

The following result is essentially well known.
%: see \cite{BL}, \cite{P}, \cite{BdL}. 
We give a proof for sake of completeness.

\begin{proposition}\label{PPP} Let $w$ be a lower Christoffel word of slope $[n_0,\ldots,n_m]$ on the alphabet $\{0<1\}$, with standard 
factorization $w=w'w''$. Then
$$
P(n_0)\cdots P(n_{m-1}) P(n_m-1)=
\begin{cases}
\left( \begin{array}{cc} |w'|_1&|w''|_1\\|w'|_0&|w''|_0 \end{array}\right) & \text{if $m$ even}; 
\\[.25in]
\left( \begin{array}{cc} |w''|_1&|w'|_1\\|w''|_0&|w'|_0 \end{array}\right) & \text{if $m$ odd}.
\end{cases}
$$
\end{proposition}

\begin{proof} We use the fact that the set of lower Christoffel words of length $\geq 2$ is the smallest set containing $01$ and 
closed under the two substitutions $G,\tilde D$, where $G(0)=0,G(1)=01,\tilde D(0)=01,\tilde D(1)=1$ (see \cite{L} or \cite {Re}). 
Moreover, these substitutions 
preserve lower Christoffel words, and hence the standard factorization.

The result is clear for $w=01$, of slope $1=[1]$, of standard factorization $w=0\cdot 1$, because $P(0)=\left( \begin{array}{cc} 
0&1\\1&0 \end{array}\right)$, and since we 
are here in the even case.

Suppose now that the result is true for $u$, of slope having the continued fraction expansion $[n_0,\ldots,n_m]$. It is enough to 
prove the proposition for $w=G(u)$ and $w=\tilde D(u)$. We do it in the first case, the second one is similar and somewhat easier 
%(the lack of symmetry comes from the lack of symmetry of continued fractions with respect to inversion).
Thus we assume that $w=G(u)$. This implies that $w'=G(u'),w''=G(u'')$, and that
$$
|w|_0=|u|_0+|u|_1, |w|_1=|u|_1,
$$
and similarly for $w',w''$.
Denote by $s(w)$ the slope of $w$. Then 
$
\frac{1}{s(w)}=\frac{|w|_0}{|w|_1}=\frac{|u|_0+|u|_1}{|u|_1}=1+\frac{1}{s(u)}=1+[0,n_0,\ldots,n_m]$.
If $n_0>0$, then $1+\frac{1}{s(u)}=[1,n_0,\ldots,n_m]$ and 
$s(w)$ has the continued fraction expansion $[0,1,n_0,\ldots,n_m]$. If $n_0=0$, then $\frac{1}{s(w)}
=1+\frac{1}{s(u)}=1+
[n_1,\ldots,n_m]=[n_1+1,\ldots,n_m]$, hence $s(w)$ has the continued fraction expansion $[0,n_1+1,\ldots,n_m]$. Note that in both cases, the continued fractions of the slopes of $u$ and $w$ have the same length parity.

Suppose that $m$ is even (the odd case is similar). Then by induction
$$
P(n_0)\cdots P(n_{m-1}) P(n_m-1)=\left( \begin{array}{cc} |u'|_1&|u''|_1\\|u'|_0&|u''|_0 \end{array}\right).
$$
Suppose that $n_0>0$; noting that $P(0)$ is the matrix of the transposition $(1,2)$, we compute
\begin{align*} 
&P(0)P(1)P(n_0)\cdots P(n_{m-1}) P(n_m-1)\\
&\quad=\left(\begin{array}{cc} 1& 0 \\ 1 & 1 \end{array}\right)P(n_0)\cdots P(n_{m-1}) P(n_m-1)\\
&\quad =\left(\begin{array}{cc} 1& 0 \\ 1 & 1 \end{array}\right)
 \left( \begin{array}{cc} |u'|_1&|u''|_1\\|u'|_0&|u''|_0 \end{array}\right)\\
&=\left( \begin{array}{cc} |u'|_1&|u''|_1\\|u'|_0+|u'|_1&|u''|_0+|u''|_1 \end{array}\right)=\left( \begin{array}{cc} |w'|_1&|w''|_1\\|w'|_0&|w''|_0 \end{array}\right),
\end{align*}
which settles this case. Suppose that $n_0=0$; we compute
\begin{align*}
&P(0)P(n_1+1)\cdots P(n_{m-1}) P(n_m-1) \\
&\quad =\left(\begin{array}{cc} 1& 0 \\ n_1+1 & 1 \end{array}\right)P(n_2)\cdots P(n_{m-1}) 
P(n_m-1)\\
&\quad =\left(\begin{array}{cc} 1& 0 \\ 1 & 1 \end{array}\right) \left(\begin{array}{cc} 1& 0 \\ n_1 & 1 \end{array}\right)
P(n_2)\cdots P(n_{m-1}) P(n_m-1)\\
&\quad =\left(\begin{array}{cc} 1& 0 \\ 1 & 1 \end{array}\right)P(0)P(n_1)
P(n_2)\cdots P(n_{m-1}) P(n_m-1)\\
&\quad =\left(\begin{array}{cc} 1& 0 \\ 1 & 1 \end{array}\right) \left( \begin{array}{cc} |u'|_1&|u''|_1\\|u'|_0&|u''|_0 \end{array}\right),
\end{align*}
and we conclude as previously.
\end{proof}

Let $g$ be a Sturmian sequence whose slope $s$ has the infinite continued fraction expansion $s=[n_0,n_1,n_2,\ldots]$. The 
slopes of the Christoffel words which are factors of $g$ are the {\em semi-convergents} of $s$; that is, the rationals of the form 
$[n_0,\ldots,n_{m-1},h]$, with $m\geq 0$, and $1\leq 
h\leq n_m$. 

After this preparation, we may translate Theorem \ref{main} in terms of continued fractions; we use the notations of this theorem.
The integer $\nu$ is defined by $|w_{\nu-1}|\leq n\leq |w_\nu|-1$. Since the slope of $w_\nu$ is a semi-convergent of $s$, $n$ is equivalently defined by the inequalities
$$m\geq 0, \quad 1\leq h\leq n_m,$$
and
$$K(n_0,\ldots,n_{m-1},h-1)+K(n_1,\ldots,n_{m-1},h-1)\leq n
$$
$$\leq 
K(n_0,\ldots,n_{m-1},h)+K(n_1,\ldots,n_{m-1},h)-1,$$
with the convention that $h-1$ is omitted if $h=1$. Note that $w_\nu$ is the lower Christoffel word of slope $[n_0,\ldots,n_{m-1},h]$. Then we have, by Proposition \ref{PPP},
$$
P(n_0)\cdots P(n_{m-1}) P(h-1)=
\begin{cases}
\left( \begin{array}{cc} |w_\nu'|_1&|w_\nu''|_1\\|w_\nu'|_0&|w_\nu''|_0 \end{array}\right) & \text{if $m$ even};
\\[.25in]
\left( \begin{array}{cc} |w_\nu''|_1&|w_\nu'|_1\\|w_\nu''|_0&|w_\nu'|_0 \end{array}\right) & \text{if $m$ odd}.
\end{cases}
$$

To finish, note that many authors consider, instead of the slope $s$ as defined here, the limit of $|w|_1/|w|$ when $|w|$ tends to infinity, $w$ factor of $g$; we call this limit the {\em Slope} of $g$, denoted $S$. These two numbers are related by
$$S=s/(1+s), s=S/(1-S), S^{-1}=1+s^{-1}.$$
It is easy to deduce that if the continued fraction of $S$ is of the form $[0,a_1,\ldots,a_m,\ldots]$, where the $a_i$ are 
positive integers, then 
$$
s=
\begin{cases}
[0,a_1-1,a_2,\ldots,a_m,\ldots], & \text{if $a_1\geq 2$}; \\
[a_2,\ldots,a_m,\ldots], & \text{if $a_1=1$}.
\end{cases}
$$

\section{Fibonacci}\label{fibo}

The Fibonacci sequence $f$ is the limit of the sequence of finite words $\Phi_\nu$, where $\Phi_1=1$, $\Phi_2=0$, and $
\Phi_\nu=\Phi_{\nu-1}\Phi_{\nu-2}$ for  $\nu\geq 3$ (\cite[\S 7.1]{AS}). The length of $\Phi_\nu$ is $F_\nu$, the Fibonacci number (with 
$F_1=F_2=1$). The slope of $f$, which is the limit of $|\Phi_\nu|_1/\Phi_\nu|_0$, is $(\sqrt{5}-1)/2$.

The corresponding sequence $(w_\nu)_{\nu\in \N}$ of lower Christoffel words is defined by $w_{-2}=0, w_{-1}=1$ and the recursion, for any natural $\nu$ 
\begin{equation}\label{Fibow}
w_\nu=\begin{cases} 
	      w_{\nu-1}w_{\nu-2} , 	 & \text{if} \,\,  \nu \,\, \text{even};  \\
	       w_{\nu-2}w_{\nu-1},        &  \text{if} \,\, \nu \,\, \text{odd}.
         \end{cases}
\end{equation} 
Indeed, it is verified that, computing in the free group on $0,1$, one has $w_\nu=0\Phi_{\nu+3}0^{-1}$ if $\nu$ is even, and $\tilde 
w_\nu=1\Phi_{\nu+3}1^{-1}$ if $\nu$ is odd; so that the limit of $|\Phi_\nu|_1/\Phi_\nu|_0$ is equal to that of $|w_\nu|_1/w_\nu|_0$.
One clearly has $|w_\nu|= F_{n+3}$. Moreover, one has $|w_\nu|_0=F_{\nu+2}, |w_\nu|_1=F_{\nu+1}$, as is easily verified.

Consider the $n$-th determinantal vector $V_n$ of $f$. Then the $\nu$ in Theorem \ref{main} is defined by $F_{\nu+2}\leq n\leq F_{\nu+3}-1$; 
consequently $i=F_{\nu+3}-1-n$. %Let $\epsilon$ be the sign of the permutation $x\mapsto F_{n+1} x$ of $(\Z/F_{n+3}\Z)^*$.

Since the standard 
factorization of $w_\nu$ is given by (\ref{Fibow}), we obtain that, up to multiplication by $\pm 1$, 
$V_n$ is the perfectly clustering word associated with the composition $(F_{\nu+1}-i,i,F_{\nu+2}-i)$ over the alphabet $\{-F_\nu,-F_{\nu-2},F_{\nu-1}\}$ if $\nu$ is even, and the 
composition $(F_{\nu+2}-i,i,F_{\nu+1}-i)$ over the alphabet $\{-F_{\nu-1},F_{\nu-2},F_\nu\}$ if $\nu$ is odd.

In particular, if $n=F_{\nu+3}-1$, the 
Sturmian determinants of $f$ of order $n$ 
take, up to the sign, only the two values 
$F_\nu,F_{\nu-1}$. And if $F_{\nu+2}\leq 
n<F_{\nu+3}-1$, they take up to the sign only 
the three values $F_\nu,F_{\nu-1},F_{\nu-2}$.

Moreover, since a perfectly clustering word begins with its smallest letter and ends with its largest, and recalling the definition of 
$V_n$ (Section \ref{determ-sturmian}), we see that the determinant of the $n$ lexicographically largest factors of length $n$ of $f$ 
is $\pm F_{\nu-1}$ if $\nu$ is even, and $\pm F_\nu$ if $\nu$ is odd; and for the $n$ smallest, it is 
$\pm F_\nu$ if $\nu$ is even, and $\pm F_{\nu-1}$ if $n$ is odd.

Recall that a factor $u$ of any Sturmian sequence on $\{0,1\}$ is called {\em (right) special} if $u0,u1$ are both factors of the 
sequence. In Proposition \ref{FF}, the description of how one passes from $G_{n+1}$ to $G_n$, and the subsequent proofs, show 
that the determinant of the factors of the sequence of length $n$, without the special factor, is the middle element of of the 
alphabet. Hence, for the Fibonacci sequence, it is $\pm F_{\nu-2}$.

\subsection{Computation of the sign}\label{sign0}

It follows from what precedes that for $n=F_{\nu+3}-1$, $V_n$ is the product of $\epsilon$ by the lower Christoffel word 
associated with the composition $(F_{\nu+1},F_{\nu+2})$ over the alphabet $\{-F_\nu,F_{\nu-1}\}$ if $\nu$ is even, and the 
composition $(F_{\nu+2},F_{\nu+1})$ over the alphabet $\{-F_{\nu-1},F_\nu\}$ if $\nu$ is odd. Here $\epsilon $ is the sign of the 
permutation ``multiplication by $F_{\nu+1}$ in $\Z/F_{\nu+3}$". It is therefore of interest to compute this sign.

Recall that the {\em Lucas numbers} $L_n$ are defined by the same recursion as the Fibonacci numbers, with the initial 
conditions $L_0=2,L_1=1$. 

%m-2=n, m=n+2

\begin{proposition}\label{sign} Let $m \geq 3$ be an integer.  Consider the permutation induced by  multiplication by $F_{m-2}$ in $\Z/F_{m}\Z$. The cyclic type of this permutation is 
\begin{itemize}
\item[(a)] $1^{ L_{m/2}}2^{(F_{m} - L_{m/2})/2}$ if $m\equiv \modd{0} {4}$;
\item[(b)] $1^{F_{m/2}}2^{(F_{m} - F_{m/2})/2 }$ if $m\equiv \modd{2} {4}$;
\item[(c)] $1^1 4^{(F_{m} -1)/4 }$ if $m\equiv \modd{1,5} {6}$;
\item[(d)] $1^2 4^{(F_{m}-1)/4}$ if $m\equiv \modd{3} {6}$;
%\item[(e)] $1^1 4^{(F_{m} -1)/4 }$ if $m\equiv \modd{5} {6}$.
\end{itemize}
Furthermore, if $s_m$ denotes the sign of this permutation, we have 
$s_m = 1$ if $m \equiv \modd{ 1,2,3,4,9,11} {12}$ and
    $s_m = -1$ if $m \equiv \modd{0,5,6,7,8,10} {12}$.
\end{proposition}

Note that $s_n=\legendre{F_{m-2}}{F_m}_Z$, the Zolotareff symbol.
We begin by a lemma.
\begin{lemma}
\leavevmode
\begin{itemize}
\item[(a)] $\gcd(F_{6k+1} -1, F_{6k+3}) = 2$.
\item[(b)] $\gcd(F_{6k+3} - 1, F_{6k+5}) = 1$.
\item[(c)] $\gcd(F_{6k-1} - 1, F_{6k+1}) = 1$.
\end{itemize}
\label{gcdlem}
\end{lemma}

\begin{proof}[Proof of Proposition \ref{sign}]
\leavevmode
\begin{itemize}
\item[(a)] If $n \equiv \modd{0} {3}$ then
both $F_{2n+1} -1$ and $F_{2n+3}$ are even, so they have a factor $2$ in common.   To see that
$\gcd(F_{2n+1} -1, F_{2n+3}) = 2$,
observe that
$(F_{2n+1} - 1)(F_{2n} - 2) - F_{2n+3}(F_{2n-2} - 1) = 4$
for all $n$.  Furthermore, if $n \equiv \modd{0} {3}$,  then $F_{2n}$ is even and $F_{2n-2}$ is odd, so we can divide both $F_{2n} - 2$ and $F_{2n-2}-1$  by $2$ to get $2$
as an integral linear combination of $F_{2n+1} - 1$ and $F_{2n+3}$.  Thus this identity holds for $n = 6k+1$.

\item[(b)] It is easy to check 
$(F_{2n+1} - 1)(F_{2n+2} -1) - F_{2n+3}(F_{2n} - 1) = 2$
for all $n$.  If $n \equiv \modd{1} {3}$, then both $F_{2n+2}$ and $F_{2n}$ are odd, so we can divide both $F_{2n+2} - 1$ and $F_{2n} - 1$
by $2$ to get an integral linear combination of $F_{2n+1}-1$ and $F_{2n+3}$ that equals $1$.
Thus this identity holds for $n=6k+3$.

\item[(c)]  It is easy to check that
$F_{2n+3} F_{2n-1} - (F_{2n+1} - 1)(F_{2n+1} + 1) = 2$
for all $n$.   If $n \equiv \modd{2} {3}$ then $F_{2n-1}$ is even and $F_{2n+1}$ is odd, so we can divide
both $F_{2n-1}$ and $F_{2n+1} + 1$ by $2$ to get an integral linear combination of 
$F_{2n+1} - 1$ and $F_{2n+3}$ that equals $1$.   Thus this identity holds for $n =6k-1$.
\end{itemize}
\end{proof}

\begin{proof}
Let $s_m$ denote the sign of the permutation induced by multiplication by $F_{m-2}$ in $\Z/F_{m}\Z$.

\bigskip\noindent{\it Case 1:} 
$m$ is even.

From the classical Fibonacci number identity $F_i^2 = F_{i-2} F_{i+2} + (-1)^i$, valid for all $i$,
we see that if $i$ is even then
$F_i^2 = F_{i-2} F_{i+2} + 1$ and so
$F_i^2 \equiv \modd{1} {F_{i+2}}$.
Setting $i= m-2$ shows that multiplication by $F_{m-2}$ induces cycles of size either $1$ or $2$.

The number $a$ will be in a cycle of length $1$ iff  $aF_{m-2} = \modd{a} {F_{m}}$,
or in other words $a(F_{m-2} - 1) \equiv \modd{0} {F_{m}}$.
Now there are two cases, namely (a) $m=4k$, and (b) $m = 4k+2$. 

\bigskip\noindent{\it Case 1a:}  $m = 4k$.

We use the identities
\begin{align*}
F_{4k-2} - 1 &= L_{2k} F_{2k-2}\\
F_{4k} &= F_{2k} L_{2k}.
\end{align*}

Suppose $a(F_{4k-2} - 1) \equiv \modd{0} {F_{4k}}$.  Then
\begin{align*}
a L_{2k} F_{2k-2} &= a(F_{4k-2} - 1)\\
		  &\equiv \modd{0} {F_{4k}}\\
        &  \equiv \modd{0} {F_{2k} L_{2k}},
\end{align*}
so cancelling $L_{2k}$ gives the congruence
$a F_{2k-2} \equiv \modd{0} {F_{2k}}$.
But $F_{2k-2}$ is relatively prime to $F_{2k}$, so it must be that 
$a$ is actually a multiple of $F_{2k}$ and every such multiple works.
There are $F_{4k}/F_{2k}  = L_{2k}$ such multiples
in the numbers mod $F_{4k}$.  So there are $L_{2k}$ cycles of length $1$ and $(F_{4k} - L_{2k})/2$ cycles of length $2$.
The sign of the permutation is then $s_m = (-1)^e$ where
$e \equiv \modd{(F_{4k} - L_{2k})/2} {2}$.   It is easy to see that
$e=0$ iff $k \equiv \modd{0} {3} $ and $e=1$ otherwise.
So $s_m = 1$ if $m\equiv \modd{4} {12}$ and $-1$ if $m\equiv \modd{0,8} {12}$.

\bigskip\noindent
{\it Case 1b:}  $m=4k+2$.

We use the identities
\begin{align*}
F_{4k} - 1 &= L_{2k-1} F_{2k+1}\\
F_{4k+2} &= F_{2k+1} L_{2k+1} .
\end{align*}
As above,
\begin{align*}
a L_{2k-1} F_{2k+1} &= a (F_{4k} - 1) \\
                    &\equiv \modd{0} {F_{4k+2}} \\
                    &\equiv \modd{0} {F_{2k+1} L_{2k+1}},
\end{align*}
and so cancelling $F_{2k+1}$ from this congruence we see
$a L_{2k-1} \equiv \modd{0} {L_{2k+1}}$.
But $L_{2k-1}$ is relatively prime to $L_{2k+1}$, so
it must be that $a$ is a multiple of $L_{2k+1}$, and every such multiple works.
There are  $F_{4k+2} /L_{2k+1} = F_{2k+1}$ such multiples in
the numbers mod $F_{4k+2}$.
So there are $F_{2k+1}$ cycles of length $1$ and 
$(F_{4k+2} - F_{2k+1})/2$ cycles of length $2$.

The sign of the permutation is then $s_m = (-1)^e$, where
$e \equiv \modd{(F_{4k+2} - F_{2k+1})/2} {2}$.
It is easy to see that $e = 0$ if $k \equiv \modd{0} {3}$
and $e=1$ otherwise.
So now we see $s_m = 1$ if $m\equiv \modd{2} {12}$, 
and $-1$ if $m\equiv \modd{6,10} {12}$.

\bigskip\noindent{\it 
Case 2:}  $m$ is odd.

We use the identity
$F_{m-2}^4 - 1 = F_{m-4} F_{m-3} F_{m-1} F_m$,
valid for all $m$.
Then $aF_{m-2}^4 = \modd{a} {F_{m}}$, so it is clear that there can only
be cycles of length $1$, $2$, or $4$.

We now show there can never be any cycles of length $2$.
If there were, as above, we would have
$a(F_{m-2}^2 - 1) \equiv \modd{0} {F_{m}}$
But since $m$ is odd we have
$F_{m-2}^2 - 1 = F_{m-4} F_{m} - 2$,
so
\begin{align*}
a(F_{m-2}^2 - 1) &= a(F_{m-4} F_{m} - 2) \\
             &\equiv \modd{-2a} {F_{m}}\\
             &\equiv \modd{0} {F_{m}},
\end{align*}
which forces $F_{m}$ to be even, which in turn forces $m= 6k+3$ and
$a = F_{6k+3}/2$.   Now $F_{6k+1}$ is odd, so $F_{6k+1} \equiv 
\modd{1} {2}$.
Multiplying through by $a$ gives
$aF_{6k+1} \equiv \modd{a} {F_{6k+3}}$, so in fact this supposed cycle of size $2$ is actually of size $1$.

\bigskip
\noindent{\it 
Case 2a:}  $m=6k+1$.  Except for $0$, the only cycles are of length $4$,
and there are $(F_{6k+1} -1)/4$ of them.

From above it suffices to show there are no cycles of length $1$.
If there were we would have
$a F_{6k-1} = \modd{a} {F_{6k+1}}$ so
$a(F_{6k-1} - 1) \equiv \modd{0} {F_{6k+1}}$.
But $\gcd(F_{6k-1} - 1, F_{6k+1}) = 1$ by Lemma~\ref{gcdlem} (c), which forces $a \equiv 0$, impossible.
So $s_m = (-1)^e$ where $e \equiv \modd{(F_{6k+1}-1)/4} {2}$.  It is easy to check 
$s_m = 1$ if $k \equiv\modd{1} {2}$ and $s_m =-1$ otherwise.  So we get
$s_m = -1$ for $m \equiv \modd{7} {12}$ 
and $s_m = 1$ for $m \equiv \modd{1} {12}$. 

\bigskip\noindent{\it 
Case 2b:}  $m=6k+3$.  

We will see there are exactly two cycles of length $1$;
the rest therefore have to be of length $4$.
Suppose $a F_{6k+1} \equiv \modd{a} {F_{6k+3}}$.  Then
$a(F_{6k+1} - 1) \equiv \modd{0} {F_{6k+3}}$.
But $\gcd(F_{6k+1} -1), F_{6k+3}) = 2$ by
Lemma~\ref{gcdlem} (a),
so the only possibilities are
$a = 0$ or $a= F_{6k+3}/2$.   Thus 
$s_m = (-1)^e$ where $e \equiv \modd{(F_{6k+3}-2)/4} {2}$.  It is now easy to check 
$s_m = 1$ in all cases.

\bigskip\noindent{
Case 2c:}  $m = 6k+5$.   

Except for $0$,
there are only cycles of length $4$,
and there are $(F_{6k+5} -1)/4$ of them.
Suppose there is another cycle of length $1$.  Then
we would have
$a F_{6k+3} = \modd{a}  {F_{6k+5}}$ so
$a(F_{6k+3} - 1) \equiv \modd{0} {F_{6k+5}}$.
But by Lemma~\ref{gcdlem} (b), we know that $\gcd(F_{6k+3} - 1, F_{6k+5}) = 1$, which forces $a \equiv 0$, impossible.

So $s_m = (-1)^e$ where $e \equiv \modd{F_{6k+5}-1)/4} {2}$.  It is easy to check 
$s_m = 1$ if $k \equiv \modd{1} {2}$ and $-1$ otherwise.  So we get
$s_m = -1$ for $n \equiv \modd{3} {12}$ and $s_m = 1$ for 
$m = \modd{9} {12}$.

Summing up, we have
$$ s_m = \begin{cases}
1, & \text{if $m \equiv \modd{ 1,2,3,4,9,11} {12}$}; \\
-1, & \text{if $m \equiv \modd{0,5,6,7,8,10} {12}$}.
\end{cases}$$
\end{proof}

\section{Postscript}
Most of the work presented in this article was done between March and June 2024.  

As we were putting the finishing touches on this paper  in mid-September 2024, we were informed about the interesting new paper of Luca Zamboni \cite{Za}, done independently, which contains some of the same results.

\end{document}